\documentclass[12pt,leqno,fleqn]{amsart}  
\usepackage{amsmath,amstext,amsthm,amssymb,amsxtra}
\usepackage{txfonts} % pxfonts txfonts 
\usepackage[T1]{fontenc}
\usepackage{lmodern}

\usepackage{euler}
\usepackage{amsfonts}
\usepackage{latexsym}

\theoremstyle{plain}
\newtheorem{thm}{Theorem}

\newtheorem{lem}[thm]{Lemma}
\newtheorem{prop}[thm]{Proposition}
\newtheorem{cor}[thm]{Corollary}

\theoremstyle{definition}
\newtheorem{defn}{Definition}[section]
\newtheorem{exmp}[defn]{Example}

\newtheorem*{rem*}{Remark}
\newtheorem{rem}[defn]{Remark}

\newcommand{\R}{\mathbb{R}}

\newcommand{\Z}{\mathbb{Z}}

\newcommand{\DC}{DC}

\DeclareMathOperator{\dist}{dist}
\DeclareMathOperator{\diam}{diam}

\DeclareMathOperator{\WLSC}{WLSC}
\DeclareMathOperator{\WUSC}{WUSC}

\numberwithin{equation}{section}

\usepackage{mathtools}
\mathtoolsset{showonlyrefs,showmanualtags} 
\def\kint_#1{\mathchoice%
          {\mathop{\kern 0.2em\vrule width 0.6em height 0.69678ex depth -0.58065ex
                  \kern -0.8em \intop}\nolimits_{\kern -0.4em#1}}%
          {\mathop{\kern 0.1em\vrule width 0.5em height 0.69678ex depth -0.60387ex
                  \kern -0.6em \intop}\nolimits_{#1}}%
          {\mathop{\kern 0.1em\vrule width 0.5em height 0.69678ex depth -0.60387ex
                  \kern -0.6em \intop}\nolimits_{#1}}%
          {\mathop{\kern 0.1em\vrule width 0.5em height 0.69678ex depth -0.60387ex
                  \kern -0.6em \intop}\nolimits_{#1}}}

\usepackage{hyperref} %,hypertexnames=false,colorlinks,[pagebackref]
\hypersetup{
    colorlinks=true,       % false: boxed links; true: colored links
    linkcolor=blue,          % color of internal links
    citecolor=magenta,        % color of links to bibliography
    filecolor=magenta,      % color of file links
    urlcolor=cyan           % color of external links
}

\allowdisplaybreaks
\swapnumbers

\newtheorem*{ack}{Acknowledgment}
\numberwithin{defn}{subsection}

\title[Fractional Hardy inequalities]{A framework for fractional Hardy inequalities}

\author[B{.} Dyda]{Bart{\l}omiej Dyda}
\address[B.D.]{Faculty of Mathematics\\ University of Bielefeld\\
Postfach 10 01 31,
D-33501 Bielefeld, Germany\\
\and
 Institute of Mathematics and Computer Science\\ Wroc{\l}aw University of Technology\\
Wybrze\.ze Wyspia\'nskiego 27,
50-370 Wroc{\l}aw, Poland
}
\email{dyda@math.uni-bielefeld.de  bdyda@pwr.wroc.pl}

\author[A.V. V\"ah\"akangas]{Antti V. V\"ah\"akangas}
\address[A.V.V.]{Department of Mathematics and Statistics,
P.O. Box 68, FI-00014 University of Helsinki, Finland} \email{antti.vahakangas@helsinki.fi}

\begin{document}

\begin{abstract}
We provide a general framework for fractional Hardy inequalities. Our framework
covers, for instance,
fractional inequalities
related to the Dirichlet forms of some L\'evy processes, and 
weighted fractional inequalities on irregular open sets.
\end{abstract}

\maketitle

\section{Introduction}
The objective of the present paper is to study inequalities of the general form
\begin{equation}\label{eq:fhi_INTRO}
 \int_D \frac{|u(x)|^p}{\phi(\delta_x)}\,\mu(dx) \leq
c \int_D\!\int_{D\cap B(x,R\delta_x)}
 \frac{|u(x)-u(y)|^p}{\phi(\delta_x)\delta_x^d}\,\mu(dy)\,\mu(dx)\,,\qquad c,R>0\,,
\end{equation}
on metric measure spaces $(X,\rho,\mu)$, % and partly
with the emphasis
on $X=\R^d$ equipped with the
 Euclidean distance and the Lebesgue measure.
We write
$\delta_x = \dist(x,X\setminus D)$ and $D\subset X$ is a possibly irregular open set.
The  
function $\phi:(0,\infty)\to (0,\infty)$ is a `perturbation' of
a power function $t\mapsto t^\eta$ for some $\eta\in\R$,
and the exponent $p$ satisfies $0<p<\infty$.

Our main result, Theorem \ref{t.main}
in \S \ref{s.main}, brings together two so-far distinct lines along which
the fractional Hardy inequality has been generalised: one
of them related to the function $\phi$, and the other to the regularity of the open set
$D\subset X$. 
Let us present the Euclidean version of our main result here; 
Theorem \ref{t.1} below is a
combination of Theorem~\ref{t.main}
and Propositions~\ref{p.fat} and~\ref{p.thin}.

\begin{thm}\label{t.1}
Let $0<p<\infty$, $H\in (0,1]$ and $\eta\in\R$.
Suppose that $D\not=\emptyset$ is a proper $\kappa$-plump open set in $\R^d$
and $\phi:(0,\infty)\to(0,\infty)$ is a function
 so that either condition (T) or condition (F) holds:
\begin{itemize}
\item[(T)]  
$ \overline{\mathrm{dim}}_A(\partial D)< d - \eta$, $D$ is unbounded,
and $\phi\in \WUSC(\eta,0,H^{-1})$;
\item[(F)]  $ \underline{\mathrm{dim}}_A(\partial D)  > d-\eta$, $D$ is bounded or $\partial D$ is unbounded, and  $\phi\in \WLSC(\eta,0,H)$.
\end{itemize}
Then inequality \eqref{eq:fhi_INTRO} holds
for all measurable functions $u$ for which the left hand side is finite.
\end{thm}

For the definitions of plumpness and Assouad dimensions $\overline{\mathrm{dim}}_A$ and $\underline{\mathrm{dim}}_A$, we refer to Section~\ref{s.notation}; the classes $\WUSC$ and $\WLSC$ are defined in Definition~\ref{d.WUSC}.
Let us remark that inequality  \eqref{eq:fhi_INTRO} fails for nonzero constant functions.
This means that under the assumptions of Theorem~\ref{t.1}, the integral $\int_D \phi(\delta_x)^{-1}\,dx$ must be divergent, on the other hand, it shows that the assumption of
the left hand side of \eqref{eq:fhi_INTRO} being finite is essential.
The relatively simple proof of our main result is a refinement of the techniques in
\cite{Dyda} where, e.g., Hardy inequalities \eqref{e.sp_hardy} with $\beta=0$
in case of bounded Lipschitz domains are established.

Theorems~\ref{t.1} and \ref{t.main} provide a general framework for fractional Hardy inequalities,
as they allow  for both general 
open sets $D$ and functions $\phi$ simultaneously. 
In the sequel, we will state separate corollaries in each of these directions to make the exposition of our framework transparent.
First, instead of considering classes $\WUSC$ and $\WLSC$, we will confine ourselves to the well-known \emph{regularly varying} functions.
%Regularly varying functions $\phi$ with index
%$\eta>0$  are allowed in \eqref{eq:fhi_INTRO}, see~\S \ref{s.function} and \cite{MR1015093}.
Let us remind that $\phi$ is called \emph{regularly varying} at origin (resp.\ infinity)  with index
$\eta$, if 
\[ 
 \frac{\phi(\lambda x)}{\phi(x)} \to \lambda^\eta
\]
 when $x\to 0_+$ (resp.\ when $x\to \infty$) for every $\lambda>0$.
%We state the following theorem as an example.
In the following corollary the geometry of the underlying domain $\R^d\setminus \{0\}$ is particularly simple.

\begin{cor}\label{t.onepoint}
Suppose that $\phi:(0,\infty)\to(0,\infty)$ is a~regularly varying function at origin of index $\rho_0$,
a~regularly varying function at infinity of index $\rho_\infty$, and is bounded and bounded away
from zero on every compact subset of $(0,\infty)$.
Suppose that either $0 <  \rho_0, \rho_\infty < d$ or $\rho_0, \rho_\infty > d$.
Let $0<p<\infty$.
Then there exists a constant $c=c(\phi, d, p)$ such that inequality
\begin{equation}\label{eq:onepoint}
 \int_{\R^d \setminus \{0\}} \frac{|u(x)|^p}{ \phi(|x|)}\,dx \leq
c \int_{\R^d \setminus \{0\}}\!\int_{\R^d \setminus \{0\}}
 \frac{|u(x)-u(y)|^p}{\phi(|x-y|)|x-y|^d}\,dy\,dx
\end{equation}
holds for every measurable function $u$ for which the left hand side is finite.
\end{cor}

A proof of this corollary can be found in \S\ref{s.proof.c2}.
Inequalities like \eqref{eq:onepoint} have been studied in
\cite{MR1731740} for weights of more general (but also more complicated) form and $p>1$,
and in \cite{MR2834782, MR1731740, MR1982932} in the one-dimensional case.
The forms appearing on the right hand side of \eqref{eq:onepoint} for $p=2$ (and for more general domains) 
are, at least for some functions $\phi$, the Dirichlet forms of certain L\'evy processes, which are
being extensively studied, see e.g. \cite{Grzywny, MR2928720, MR2240700} and \cite[Section 4.1]{BogdanGrzywnyRyznar}.

To discuss our results for irregular open sets, we confine
ourselves to weighted fractional Hardy inequalities  in $\R^d$, i.e.,
we consider the function $\phi(t)= t^{sp-\beta}$ with
$d+sp\ge 0$, in which case inequality \eqref{eq:fhi_INTRO} yields
\begin{equation}\label{e.sp_hardy}
 \int_D \frac{|u(x)|^p}{\delta_x^{sp}}\,\delta_x^{\beta}\,dx \leq 
c\int_D
 \int_D \frac{\lvert u(x)-y(y)\rvert^p}{\lvert x-y\rvert^{d+sp}}\, \delta_x^{\beta}\,dy\, dx\,.
\end{equation}
An open set  $D\subset \R^d$ is said to admit {\em  $(s,p,\beta)$-Hardy inequality}, if 
inequality \eqref{e.sp_hardy} holds
for all functions $u\in C^\infty_0(D)$
(i.e., smooth with compact support in $D$)
 with $c$ independent of $u$.
There has been recent interest in $(s,p,0)$-Hardy inequalities
in connection with the boundary regularity of an open set $D$, we refer to 
 \cite{E-HSV,ihnatsyeva3,ihnatsyeva1,ihnatsyeva2}.
In another direction,  the sharp constants for fractional Hardy-type
inequalities on general domains are obtained in \cite{MR2659764}, where the distance
is replaced with an averaged pseudo distance.
In \cite{MR2910984} these results are further refined;
let us also mention the other related papers  \cite{MR2663757,MR2755892,MR3021545,MR2723817,
MR2823046}.

The non-fractional counterpart of inequality \eqref{e.sp_hardy} has also been studied.
Namely, the following weighted
$(p,\beta)$-Hardy inequality, with $c>0$,
\begin{equation}\label{e.p-hardy}
\int_D \frac{\lvert u(x)\rvert^p}{\delta_x^p}\,\delta_x^{\beta}\,dx 
\le c \int_D \lvert \nabla u(x)\rvert^p\,\delta_x^\beta\,dx
\end{equation}
holds for every $u\in C^\infty_0(D)$ if $D$ is a bounded Lipschitz domain,
$1<p<\infty$, and $\beta<p-1$, \cite{MR0163054}.
More generally, an open set admits a $(p,\beta)$-Hardy inequality if the 
complement $D^c=\R^d\setminus D$
is either sufficiently `thin' or  `fat'.
For instance, an open set $D$ admits a $(p,0)$-Hardy inequality
if $D^c$ is $(1,p)$-uniformly fat and $1<p<\infty$,
\cite{Lewis1988}. 
The $(1,p)$-fatness of
$D^c$ is also known to be sufficient for certain $(p,\beta)$-Hardy inequalities,
\cite{LeLip,MR1010807}.
A deeper understanding of the `thin' vs. `fat'  dichotomy is reached
in an independent recent study \cite{lehrbackHardyAssouad}, where an open set $D\subset X$
is shown to admit
 a $(p,\beta)$-Hardy inequality if $D^c=X\setminus D$
 sufficiently thin or fat, measured in
terms of upper and lower Assouad dimension ($\overline{\mathrm{dim}}_A$ and  $\underline{\mathrm{dim}}_A$), respectively.
We also refer to \cite{MR1948106}.

Our framework covers an  Assouad dichotomy result for
fractional  $(s,p,\beta)$-Hardy inequalities with $X=\R^d$,
see the following Corollary~\ref{t.cor}.

%The following is our Assouad dichotomy result. For the relevant
%definitions, we refer to \S \ref{s.notation}.
\begin{cor}\label{t.cor}
Let $p$, $s$, $\beta$ be real numbers so that $0<p<\infty$ and  
$d+sp\ge 0$.
Suppose $D\not=\emptyset$ is a proper $\kappa$-plump open set in $\R^d$ so that either condition (T) or condition (F) holds:
\begin{itemize}
\item[(T)] $\overline{\mathrm{dim}}_A(\partial D) < d-sp+\beta$ and $D$ is unbounded;
\item[(F)]  $\underline{\mathrm{dim}}_A(\partial D)> d-sp+\beta$, and  $D$ is bounded or $\partial D$ is unbounded.
\end{itemize}
Then $D$ admits an $(s,p,\beta)$-Hardy inequality, i.e., inequality \eqref{e.sp_hardy} holds
for all $u\in C_0^\infty(D)$ with a constant $c>0$ independent of $u$.
\end{cor}

This corollary follows  from Theorem~\ref{t.1} with the aid of Example~\ref{e.reg}.
As an
illustration of this result,
 we may consider the
domain
$D\subset \R^2$
bounded by the Koch snowflake.
 It is a
 %bounded
 domain with a property
$\underline{\mathrm{dim}}_A(\partial D) = \log 4 / \log 3$.
Since $D$ is also
$\kappa$-plump, Corollary~\ref{t.cor} does apply.
In the `thin case' we may, e.g., consider the unbounded domain $G := \R^d\setminus \overline{D}$. Now $G$ is $\kappa$-plump
and it satisfies $\overline{\mathrm{dim}}_A(\partial G) = \log 4/ \log 3$.
Let us note that \cite[Theorem 1.1]{Dyda}, apart from the case (T2) for $\alpha>1$, is a special case of Corollary~\ref{t.cor}.
We would also like to note that we do not need to assume the positivity of $s$ in Corollary \ref{t.cor}.

Let us comment on the cases (T)  and (F) in Corollary~\ref{t.cor}.
Focusing on the case (T) first, 
recall that $\overline{\mathrm{dim}}_A(\partial D)=d-1$ for a Lipschitz domain $D$.
The unboundedness of $D$ cannot be removed, at least if $0<s<1$, in which case
a bounded Lipschitz domain satisfies an $(s,p,0)$-Hardy inequality
if and only if $sp>1$, \cite{Dyda}. 
Certain non-homogeneous $(s,p,0)$-Hardy inequalities
remain valid for John domains $D$
with $\overline{\mathrm{dim}}_{A}(\partial D) < d-sp$, \cite{ihnatsyeva2}. 
Therein (T) with $\beta=0$ is formulated  in terms of a certain Aikawa dimension which equals to the upper Assouad dimension in Euclidean spaces, see \cite{lehrbackII}.
Recalling that John domains are both bounded and $\kappa$-plump, we may conclude
that our framework provides a far-reaching generalisation of the
aforementioned non-homogeneous results to the  case
of unbounded open sets.

%\color{blue}
%\sout{
%The dimensional restriction in (T) is somewhat natural: under some {\em a priori} conditions on $D$, 
%the inequality $\overline{\mathrm{dim}}_A(\partial D) < d-sp$
%is equivalent with  non-homogeneous $(s,p,0)$-Hardy inequality on 
%Triebel--Lizorkin  spaces,  \cite{ihnatsyeva2}.
%Likewise, 
%the $\kappa$-plumpness condition is natural in some cases.
%For illustration, let $D=\R^d\setminus K$, where $K$ is a closed set
%such that $\partial D=K$. Now
% the $\kappa$-plumpness of $D=\R^d\setminus K$ with some $\kappa\in (0,1)$
%is characterised
%by inequality 
%(T) with $d-sp+\beta= d$, i.e.,
%\[\overline{\mathrm{dim}}_A(K)=\overline{\mathrm{dim}}_A(\partial D) < d\,.\] 
%We refer to \cite{Luukkainen} for further  results
%on %so-called porosity and
% the upper Assouad dimension.}
% \color{black}

Moving on to the case (F) with  `fat' boundary, let us  first formulate an
illustrative, but more restrictive,
version %corollary
 of Corollary~\ref{t.cor}. We refer to \S \ref{s.unif_fatness}
for the relevant definitions.

\begin{cor}\label{t.fat}
Let $p$, $s$, $\beta$ be real numbers so that $1<p<\infty$, $0<sp-\beta<d$, and 
$d+sp\ge 0$.
Suppose $D$ is a  $\kappa$-plump
 open set in $\R^d$ such that $\partial D$ is $(s-\beta/p,p)$-uniformly fat
(-locally uniformly fat, if $D$ is bounded).
 Then $D$ admits an $(s,p,\beta)$-Hardy inequality.
\end{cor}

This corollary is a consequence of 
Corollary~\ref{t.cor} and Propositions \ref{p.unif_fatness} and \ref{p.local_unif_fatness}.
Unlike in the case of inequality \eqref{e.p-hardy} with $\beta=0$ and $s=1$,
 the $(s,p)$-uniform fatness of $\partial D$ (let alone $D^c$)
is not a sufficient condition for an open set $D$ to admit an $(s,p,0)$-Hardy inequality
(at least) in the case of $0<sp\le 1$.
This `non-local obstruction' is recognised and addressed in  \cite {ihnatsyeva3}.
It affects the fractional Hardy inequalities studied in \cite{E-HSV}, where
 $D^c$ is assumed to be $(s,p$)-uniformly fat and,
as a conclusion, on the right hand side of \eqref{e.sp_hardy}
one then has integration over $\R^d\times \R^d$.

Suppose that $D$ is an open set whose boundary is $(s,p)$-uniformly fat  (locally uniformly fat, if $D$ is bounded). It is an interesting question, what additional conditions are sufficient for $D$
to admit an $(s,p,0)$-Hardy inequality. To this end, we improve 
\cite[Corollary 1.4]{ihnatsyeva3} where uniformity, \cite{Martio,MR927080}, of a domain $D$ is
shown to be a sufficient additional condition. Indeed, by Corollary~\ref{t.fat},
 we may replace uniformity with the significantly weaker $\kappa$-plumpness.
Let us nevertheless remark that \cite[Theorem 4.1]{ihnatsyeva3},
stated in terms of a `visibility condition on the boundary', still covers some
other cases where our results do not apply, e.g., certain domains with outward cusps.

The structure of this paper is as follows.
In \S \ref{s.notation} we define
both the lower and upper Assouad dimension, and the notion
of $\kappa$-plumpness. We also present other basic notation.
Our main result is Theorem \ref{t.main}, stated and proven in \S \ref{s.main}.
There
we also define
classes $\WLSC$ and $\WUSC$ of functions $\phi$ and a condition
$\DC(a,\gamma,d)$ for open sets $D$. 
The latter condition is further clarified in 
\S \ref{s.fat} and \S \ref{s.thin}, where
we study the cases of `fat' and `thin' boundaries in terms of 
uniform fatness, and the lower
and upper Assouad dimension.

{\ack{Research is supported by the DFG through SFB-701 `Spectral Structures and Topological Methods in Mathematics'.
Part of the research was done
while the second author was visiting University of Bielefeld,
and he
would like to thank B. Dyda and M. Ka{\ss}mann for their hospitality.
The first author was supported in part by the NCN grant 2012/07/B/ST1/03356.
The authors would like to thank  K.~Bogdan, T.~Grzywny and J.~Lehrb\"ack for helpful discussions
and preprints of \cite{BogdanGrzywnyRyznar} and \cite{lehrbackHardyAssouad}.
}}

\section{Assouad dimensions and plumpness}\label{s.notation}

We  
recall the lower and upper Assouad dimensions of a set $\emptyset\not=E\subset \R^d$, \cite{KLV}.
The lower Assouad dimension measures the `fatness' of a set $E$,
whereas the upper one measures how `thin'  a set $E$ is.
The upper Assouad dimension is often called {\em Assouad dimension}, a notion tracing back to \cite{Assouad} and even \cite{Bouligand}.
We refer to \cite{KLV,Luukkainen} for further information and other results.

\begin{defn}
Consider all $\lambda\ge 0$ for which there is $C> 0$ so that,
if $0<r<R<2\diam(E)$ and $x\in E$, then
at least $C(R/r)^\lambda$ balls---centred in $E$ and of radius $r$---are needed to cover $B(x,R)\cap E$. The supremum of all such $\lambda$ is 
called the {\em lower Assouad dimension of $E$} and it is
denoted by $\underline{\mathrm{dim}}_{A}(E)$.
\end{defn}

\begin{defn}\label{def:uAssouad}
Consider all $\lambda\ge 0$ for which there is $C> 0$ so that,
if $0<r<R<2\mathrm{diam}(E)$ and $x\in E$, then we can 
cover $E\cap B(x,R)$  by 
at most $N\le C(R/r)^\lambda$ balls $B_1,\ldots,B_N$ such
that each $B_j$ is centred in $E$ and has radius $r$. 
We call the infimum of all such $\lambda$ the {\em upper Assouad dimension
of $E$}, and write it as $\overline{\mathrm{dim}}_A(E)$.
\end{defn}

We also
recall a geometric notion from  \cite{MR927080}. See also \cite{mv}.
\begin{defn}
A set $A\subset \R^d$ is {\em $\kappa$-plump}
with $\kappa\in (0,1)$ if, for each $0<r< \mathrm{diam}(A)$ and each $x\in \bar{A}$, there
is $z\in \bar B(x,r)$ such that
$B(z,\kappa r)\subset A$.\end{defn}

Here is other notation;
$(X,\rho,\mu)$ is a metric measure space, and we denote $\delta_x=\dist(x,D^c)$ with $D^c=X\setminus D$.
The open ball centred at $x\in X$ and of radius $r>0$ is denoted by
$B(x,r)\subset X$.
The boundary of set $A$ is written
as $\partial A$,
$\overline{A}$ denotes the closure of $A$,
 and $\lvert A\rvert$ is the Lebesgue measure of a
measurable set $A\subset \R^d$.
For a proper open set $D\subset \R^d$, we fix its
Whitney decomposition $\mathcal{W}(D)$, and write
$\mathcal{W}_m(D)$
for the family of Whitney cubes with side length $2^{-m}$, $m\in\Z$.
If $Q\in\mathcal{W}(D)$, then
\begin{equation}\label{dist_est}
  \diam(Q)\le \dist(Q,\partial D)\le 4\diam(Q)\,.
\end{equation}
For other properties of Whitney cubes we refer to \cite[VI.1]{MR0290095}.

\section{Main result}\label{s.main}

We state and prove our main result.
For definition of conditions
$\DC(a,\gamma,d)$, $\WLSC(\eta,0,H)$ and $\WUSC(\eta,0,H)$, we refer to
\S \ref{s.open_set} and \S \ref{s.function}. 
The proof of Theorem \ref{t.main} is taken up in \S \ref{s.main_sub}.

\begin{thm}\label{t.main}
Suppose that a proper open set $D\subset X$ satisfies $\DC(a,\gamma,d)$ 
with $a\in (0,\infty)\setminus \{1\}$.
Moreover, suppose
 that for some $H\in (0,1]$, either
$a\in (0,1)$, $\eta+\gamma-d>0$ and $\phi\in \WLSC(\eta,0,H)$,
or 
$a>1$, $\eta+\gamma-d<0$ and $\phi\in \WUSC(\eta,0,H^{-1})$.
 Then for any $0<p<\infty$ there exist constants $c$ and $R>0$ such that
\begin{equation}\label{eq:fhi}
 \int_D \frac{|u(x)|^p}{\phi(\delta_x)}\,\mu(dx) \leq
c \int_D\!\int_{D\cap B(x,R\delta_x)}
 \frac{|u(x)-u(y)|^p}{\phi(\delta_x)\delta_x^d}\,\mu(dy)\,\mu(dx)
\end{equation}
for all measurable functions $u$ for which the left hand side is finite.
\end{thm}

\subsection{Assumptions on a function $\phi$}\label{s.function}
We adopt the notion of a global weak lower (or upper) scaling condition ($\WLSC$ or $\WUSC$ for short)
 from \cite[Section 3]{BogdanGrzywnyRyznar}.
In our case, the middle parameter in $\WLSC(\cdot, 0,\cdot)$ and $\WUSC(\cdot, 0,\cdot)$ is always zero and could be therefore omitted, but we prefer to keep the original notation.
We formulate these conditions in an equivalent way, which is more convenient for our purposes than the original formulation.
\begin{defn}\label{d.WUSC}
Let $\eta\in \R$ and $H\in (0,1]$.
We say that a function $\phi:(0,\infty)\to(0,\infty)$ satisfies global $\WLSC(\eta, 0, H)$
(resp.,  $\WUSC(\eta, 0, H^{-1})$)
and write $\phi \in \WLSC(\eta, 0, H)$ ($\phi \in \WUSC(\eta, 0, H^{-1})$), if
\begin{align}
\phi(st) \geq H t^\eta \phi(s), \quad s>0\,, \label{eq:phi-assumption}
\end{align}
for every $t\geq 1$ (resp., for every $t\in (0,1]$).
\end{defn}

\begin{rem}
If the domain $D$ in Theorem~\ref{t.main} is bounded, then it suffices to assume \eqref{eq:phi-assumption}
for all $s, st < \diam(D)$.
\end{rem}

\begin{exmp}\label{e.reg}
Function $\phi(x)=x^\eta$, $\eta \in \R$, satisfies $\WLSC(\eta,0,1)$ and $\WUSC(\eta,0,1)$.
\end{exmp}

\begin{exmp}\label{ex:regular}
Suppose that $\phi:(0,\infty)\to(0,\infty)$ is a~regularly varying function at origin of index $\rho_0$,
a~regularly varying function at infinity of index $\rho_\infty$, and is bounded and bounded away
from zero on every compact subset of $(0,\infty)$.
If $\rho_0>\eta$ and $\rho_\infty>\eta$, then $\phi \in \WLSC(\eta,0,H)$ for some $H\in(0,1]$, and
if $\rho_0<\eta$ and $\rho_\infty<\eta$, then $\phi \in \WUSC(\eta,0,H^{-1})$ for some $H\in(0,1]$.
These follow from Potter's theorem \cite[Theorem 1.5.6]{MR1015093}.

We note that if, say, $a<1$, $\rho_0\geq \eta$, $\rho_\infty \geq \eta$, and if $\eta+\gamma-d>0$
and the assumptions on domain in Theorem~\ref{t.main} hold, then also the assertion \eqref{eq:fhi} holds.
Indeed, for every $\varepsilon>0$ function $\phi$ satisfies $\WLSC(\eta-\varepsilon,0,H_\varepsilon)$ with
some constant $H_\varepsilon \in (0,1]$,
hence by taking $\varepsilon>0$ small enough we still have $(\eta-\varepsilon)+\gamma-d>0$. 

To have more concrete examples, let us note that functions
\begin{align*}
\phi_1(x) &= x^\alpha + x^\beta,\\
\phi_2(x) &= x^\eta(1+|\log x|)^\beta
\end{align*}
are regularly varying both at the origin (of indices $\min(\alpha,\beta)$ and $\eta$, respectively)
and at infinity  (of indices $\max(\alpha,\beta)$ and $\eta$, respectively).
\end{exmp}

\begin{exmp}
Functions $\phi$ are not confined to regularly varying functions. Indeed,
$\phi(x) = x^\eta e^x$ satisfies $\WLSC(\eta,0,1)$, but is not regularly varying at infinity.
\end{exmp}

\subsection{Assumption $\DC(a,\gamma,d)$ on open sets}\label{s.open_set}
In what follows we assume that $D$ is an open set in  a metric measure
space $(X,\rho,\mu)$.
%We denote $\delta_x=\dist(x,X\setminus D)$ for $x\in X$.

\begin{defn}\label{d.bsets}
We say that $D$ satisfies a {\em domain condition} $\DC(a,\gamma,d)$
where $\gamma\in\R$, $d>0$, $a>0$,
$a\neq 1$,
 if there exist $M>0$
and (possibly empty) families $\mathcal{B}^{(n)}=\{B_j^{(n)}\}$ of  subsets  of $D$ indexed by $n\in\Z$  such that the following conditions (B1)--(B4) hold.
\begin{itemize}
\item[(B1)]
$D=\cup_{j,n} B_j^{(n)}$ and each $x\in D$ belongs to at most $M$ sets
$B_j^{(n)}$.
\item[(B2)]
For any $B_j^{(n)}$ we have
\begin{align*}
&M^{-1} a^n  \leq  \delta_x \leq M a^n,\quad x\in B_j^{(n)},\\
&M^{-1} a^{nd}  \leq \mu(B_j^{(n)}) \leq M a^{nd}.
\end{align*}
\item[(B3)]
For any $B_j^{(n)}$ and any integer $k>M$, 
there exists a nonempty finite set $V(B_j^{(n)}, k)$ of
indices so that, for each $i\in V(B_j^{(n)}, k)$,
\begin{align*}
\sup\{\rho(x,y)\,:\, x\in B_j^{(n)}\text{ and } y\in B_i^{(n+k)}\} &\leq
\begin{cases}
  Ma^n\,, & \text{if $a<1$\,,}\\
  Ma^{n+k}\,, & \text{if $a>1$\,.}
\end{cases}
\end{align*}
\item[(B4)] For each $n\in\Z$ and $k>M$,
\[
\sup_i \sum_{j:i\in V(B^{(n)}_j,k)} \frac{1}{\sharp V(B^{(n)}_j,k)} \le M a^{k \gamma}\,.
\]
\end{itemize}
\end{defn}

\begin{rem}
The slightly technical Definition \ref{d.bsets} allows
a unified treatment of fractional Hardy inequalities in different cases.
In the Euclidean spaces, we usually take Whitney cubes of roughly the same size as the sets $B_j^{(n)}$ and then 
conditions $(B1)$ and $(B2)$
are immediate, we refer to \S\ref{s.a_sets} and \S\ref{s.b_sets} for examples.
Formulating the definition using these general sets rather than
Whitney cubes gives us some flexibility, see Examples~\ref{ex:onepoint-thin}
and~\ref{ex:onepoint-fat} for an illustration. 
%[We could also mention the fact
%that we use WC of different generations as $\mathcal{B}^{(n)}$, but I have not
%come up with a~good way of describing that.]
The last two conditions (B3) and (B4) specify the relationship 
between the layers 
\[D_n:=D\cap \{x\, :\, M^{-1}a^n \leq \delta_x \leq Ma^n\}\]
and \[
D_{n+k}:=D\cap \{x \, :\, M^{-1}a^{n+k} \leq \delta_x \leq Ma^{n+k}\}
\] for large values of $k$.
Namely, in a `neighbourhood' of each $B_j^{(n)}\subset D_n$ there should be `sufficiently many' sets $B_i^{(n+k)}\subset D_{n+k}$; the number $\gamma$ describes this  quantitatively.
\end{rem}

Below we provide some illustrative examples of
 sets satisfying condition $\DC(a,\gamma,d)$. In the two examples $X=\R^d$ with
the Euclidean distance, in which
case $\delta_x = \dist(x,\partial D)$ for all $x\in D$. Moreover,
$\mu$ is the Lebesgue measure.

\begin{exmp}\label{ex:onepoint-thin}
Set $D=\R^d\setminus\{0\}$ satisfies condition $\DC(a,\gamma,d)$ with
$a=2$, $\gamma=0$ and 
\[M=2\vee (1-2^{-d})|B(0,1)| \vee
\frac{1}{(1-2^{-d})|B(0,1)|}\,.\]
Indeed, one may take $B_1^{(n)} := B(0,2^n)\setminus B(0,2^{n-1})$.
That is, for each $n$ there is exactly one set $B_j^{(n)}$, namely
one with $j=1$. Then $V(B_j^{(n)},k)=\{1\}$ in (B3).
\end{exmp}

\begin{exmp}\label{ex:onepoint-fat}
Set $D=\R^d\setminus\{0\}$ satisfies condition $\DC(a,\gamma,d)$ with
$a=\frac{1}{2}$, $\gamma=0$ and 
\[M=2\vee (1-2^{-d})|B(0,1)| \vee
\frac{1}{(1-2^{-d})|B(0,1)|}\,.\]
Indeed, one may take $B_1^{(n)} := B(0,2^{-n})\setminus B(0,2^{-n-1})$
and $V(B_j^{(n)},k)=\{1\}$ in (B3).
\end{exmp}

\subsection{Proof of Theorem \ref{t.main}}\label{s.main_sub}
Let us write 
\[q= 2^{p+1} M^{4+2|\eta|}  H^{-1}a^{k(\eta+\gamma-d)}\,,\quad 
R=1+M^2(1\vee a^k)\,,\quad S=2^{p+1}a^{-kd}M^{d+1}\,,\]
where $k>M$ is chosen such that $q<1$ and $a^k \vee a^{-k} > M^2$.

We fix a function $u$ for which the left hand side of \eqref{eq:fhi}
is finite, and define a set
\[
 F=\bigg\{x\in D\,:\, |u(x)|^p > S\delta_x^{-d}\int_{D\cap B(x,R\delta_x)} |u(x)-u(y)|^p\,\mu(dy)\bigg\}\,.
\]
Let us first observe that, for $x\in D\setminus F$,
\begin{equation}\label{eq:onFc}
\frac{|u(x)|^p}{\phi(\delta_x)} \leq
S \int_{D\cap B(x,R\delta_x)} \frac{|u(x)-u(y)|^p}{\phi(\delta_x)\delta_x^d}\,\mu(dy)\,.
\end{equation}
Note that if the set $F$ were empty, we would be already done.

At this stage we fix $n$ and claim that, for $x\in F\cap B_j^{(n)}$ and $i\in V(B_j^{(n)}, k)$, we
have
\begin{equation}\label{eq:claim}
\mu\left(\left\{ y\in B_i^{(n+k)} : \frac{1}{2} |u(x)| \leq |u(y)| \leq
\frac{3}{2} |u(x)| \right\}\right) \geq \frac{1}{2}\mu( B_i^{(n+k)})\,.
\end{equation}
Suppose \eqref{eq:claim} fails.
By our choice of $R$ and conditions (B2) and (B3),
$B_i^{(n+k)} \subset D\cap B(x,R\delta_x)$.
Thus, we have
\begin{align*}
\int_{D\cap B(x,R\delta_x)} |u(x)-u(y)|^p\,\mu(dy) &\geq
\int_{B_i^{(n+k)}} |u(x)-u(y)|^p\,\mu(dy) \\
&\geq \frac{1}{2} \mu( B_i^{(n+k)}) \cdot 2^{-p} |u(x)|^p\\
&\geq  2^{-p-1}a^{kd} M^{-d-1} \delta_x^d |u(x)|^p\\
&= S^{-1} \delta_x^d |u(x)|^p\,,
\end{align*}
which contradicts $x\in F$. Thus inequality \eqref{eq:claim} holds as claimed.

Let us record the following estimates for $B_i^{(n+k)}\in\mathcal{B}^{(n+k)}$
and $B_j^{(n)}\in \mathcal{B}^{(n)}$.
By condition (B2),  $\mu(B_j^{(n)}) \leq M^2 a^{-kd} \mu( B_i^{(n+k)})$, moreover,
for $x\in B_j^{(n)}$ and $y\in B_i^{(n+k)}$ it holds $M^2a^{-k}\delta_y \geq \delta_x \geq M^{-2}a^{-k}\delta_y$.
Hence, by condition \eqref{eq:phi-assumption}
\[
 \phi(\delta_x) = \phi\left(\delta_y \frac{\delta_x}{\delta_y}\right)
 \geq H \left(\frac{\delta_x}{\delta_y}\right)^\eta  \phi(\delta_y)  
 \geq H M^{-2|\eta|} a^{-k\eta}  \phi(\delta_y).
\]
Here we need to ensure that $\frac{\delta_x}{\delta_y} < 1$ in the case when $a>1$
and that $\frac{\delta_x}{\delta_y} > 1$ in the case when $a<1$. 
But these are satisfied since, by assumption,
$a^k \vee a^{-k} > M^2$, i.e., $k$ is large enough.
By the above estimate and inequality \eqref{eq:claim} we obtain
\begin{align*}
\int_{F\cap B_j^{(n)}} \frac{|u(x)|^p}{\phi(\delta_x)}\,\mu(dx) &\leq
\mu( B_j^{(n)}) \sup_{x\in F\cap B_j^{(n)}}
  \frac{|u(x)|^p}{\phi(\delta_x)}\\
&\leq \frac{2^{p+1} M^2 a^{-kd}}{ \sharp V(B_j^{(n)}, k) }
\sum_{i\in V(B_j^{(n)}, k) }
\int_{ B_i^{(n+k)}} \frac{|u(y)|^p}{H M^{-2|\eta|} a^{-k\eta}  \phi(\delta_y)}\,\mu(dy)\,.
\end{align*}
By summing over all $j$ and applying condition (B4), 
\begin{align*}
\sum_j \int_{F\cap B_j^{(n)}} \frac{|u(x)|^p}{\phi(\delta_x)}\,\mu(dx) &\leq
2^{p+1} M^{2+2|\eta|} H^{-1} a^{k(\eta-d)} 
\times \sup_i \sum_{j:i\in V(B^{(n)}_j,k)} \frac{1}{\sharp V(B^{(n)}_j,k)}\\
&\times \sum_i
\int_{ B_i^{(n+k)}} \frac{|u(y)|^p}{ \phi(\delta_y)}\,\mu(dy)\\
&\leq 2^{p+1} M^{3+2|\eta|}  H^{-1} a^{k(\eta+\gamma-d)} \sum_i
\int_{ B_i^{(n+k)}} \frac{|u(y)|^p}{ \phi(\delta_y)}\,\mu(dy)\,,
\end{align*}
and after summing over all $n$
\begin{align*}
\int_F\frac{|u(x)|^p}{\phi(\delta_x)}\,\mu(dx) &\leq
q\int_D \frac{|u(y)|^p}{ \phi(\delta_y)}\,\mu(dy)\,.
\end{align*}
Recall that $q<1$. Hence,
 by finiteness of the left hand side of \eqref{eq:fhi},
\begin{align*}
\int_F\frac{|u(x)|^p}{\phi(\delta_x)}\,\mu(dx) &\leq
\frac{q}{1-q} \int_{D\setminus F} \frac{|u(y)|^p}{\phi(\delta_y)}\,\mu(dy)\,.
\end{align*}
This estimate and inequality \eqref{eq:onFc} finish the proof.
\qed

\subsection{Proof of Corollary \ref{t.onepoint}}\label{s.proof.c2}
We use Potter's theorem \cite[Theorem 1.5.6]{MR1015093} to replace $\phi(\delta_x)$ by
 $c \phi(|x-y|)$ in the denominator, with $c=c(R,\phi)$. 
 The assumption $0<\rho_0,\rho_\infty$ is used here.
The result follows now from Theorem~\ref{t.main} and examples~\ref{ex:regular}, \ref{ex:onepoint-thin}
and~\ref{ex:onepoint-fat}.\qed

%%%%%%%%%%%%%%%%%%%%%%%%%%%%%% SECTION  SECTION SECTION
%%%%%%%%%%%%%%%%%%%%%%%%%%%%%% SECTION  SECTION SECTION 
\section{Fat boundary}\label{s.fat}

We 
prove a domain condition while assuming that the
boundary of the open set is sufficiently `fat' in terms of the lower Assouad dimension.
Then we study the relation
between lower Assouad dimension and uniform fatness.

\begin{prop}\label{p.fat}
Suppose $D\not=\emptyset$ is a proper $\kappa$-plump open set in $\R^d$ such that
$D$ is bounded or $\partial D$ is unbounded.
Then $D$ satisfies $\DC(a,\lambda,d)$ if $a=1/2$ and 
either $0<\lambda< \underline{\mathrm{dim}}_A(\partial D)$
 or $\lambda=0$.
Moreover, the associated constant $M$ depends only on $d$, $\kappa$, $\lambda$ and the constant $C$ appearing in (F1) below.
\end{prop}

Under the assumptions of Proposition \ref{p.fat}, the following two
conditions (F1) and (F2) hold.

\begin{itemize}
\item[(F1)]
There is a constant $C> 0$ as follows.
Let $0<r<R< 2\mathrm{diam}(\partial D)$ and $x\in \partial D$.
Suppose that $B_1,\ldots,B_N$ is a cover of $B(x,R)\cap \partial D$ by
balls $B_j=B(\omega_j,r)$ with $\omega_j\in \partial D$ for $j=1,\ldots,N$.
Then $N\ge  C(R/r)^\lambda$.
\item [(F2)]for each $0<r< \mathrm{diam}(D)$ and each $x\in \partial D$, there
is $z\in \bar B(x,r)$ so that
$B(z,\kappa r)\subset D$.
\end{itemize}

\subsection{Construction of families $\mathcal{B}^{(n)}$}\label{s.a_sets}
We define a constant
\begin{equation}\label{e.tau_def}
\tau=\bigg(\frac{15\sqrt d}{\kappa}\bigg)^{d}>1\,.
\end{equation}
For a given $n\in \Z$ and $a\in \{\frac{1}{2},2\}$, we define
\[
\mathcal{B}^{(n)}:=\mathcal{B}^{(n)}_{1/2}\,,\qquad \mathcal{B}^{(n)}_{a}:=\{B^{(n)}_j\} := \{Q\in\mathcal{W}(D)\,:\, \tau^{-1}\le a^{-nd} \lvert Q\rvert \le \tau\}\,.
\]
Recall that $\mathcal{W}(D)$ stands for a Whitney decomposition of $D$.
In particular, by inequalities \eqref{dist_est}, for any $x\in B^{(n)}_j\in \mathcal{B}^{(n)}_a$,
\[
\tau^{-1/d} a^{n} \le \delta_x = \mathrm{dist}(x,\partial D) \le 5\sqrt d \tau^{1/d} a^{n}\,.
\]
Observe also that a given Whitney cube $Q\in\mathcal{W}(D)$ may belong to at most $1+2d^{-1}\log_2 \tau$ families 
$\mathcal{B}^{(n)}=\mathcal{B}^{(n)}_a$ indexed by $n\in \Z$.
Let us denote by $x_j^{(n)}$ the midpoint of $B_j^{(n)}$.
For later purposes we fix, once and for all, any point $y_j^{(n)}\in\partial D$ for which
\[
\lvert x_j^{(n)} - y_j^{(n)}\rvert =\mathrm{dist}(x_j^{(n)}, \partial D)\,.
\]

\subsection{Families $V(B^{(n)}_j,k)$ for $k$ large}
If $D$ is unbounded, we  construct families $V(B^{(n)}_j,k)$ for $k>3$. If $D$ is bounded, then we construct
these families for $k>3\vee \log_2(5\tau^{1/d})$.

Let us fix $B^{(n)}_j\in \mathcal{B}^{(n)}$, and define $E:=B(y_j^{(n)},2^{-n})\cap \partial D$.
By the 
$5r$-covering theorem, see for instance \cite[p. 23]{MR1333890}, there are points
$\omega_1,\ldots,\omega_N\in E$
such that the balls $B_m:=B(\omega_m, 2^{-n-k})$ are disjoint and
$E$ is covered by the union of balls $5B_m$, $m=1,\ldots,N$.
 Let us estimate
the number $N=:N_j^{(n,k)}$ of these balls;

\begin{lem}
We have $N_j^{(n,k)}\ge C5^{-\lambda}\tau^{-\lambda/d}2^{k\lambda}$.
\end{lem}

\begin{proof}
First consider the case when $D$ is unbounded.
Since $k>3$, we find that 
\[
r:=5\cdot 2^{-n-k} < 2^{-n}=:R\,.
\]
Recall that the balls $5B_m = B(\omega_m, r)$ cover the set $E=B(y^{(n)}_j,R)\cap \partial D$.
By condition (F1), we find
that $N\ge C(R/r)^\lambda = C5^{-\lambda} 2^{k\lambda}$.
The bounded case is similar, and
we use the facts that $k>3\vee \log_2(5\tau^{1/d})$ and $\diam(\partial D)\ge \diam(D)$.
\end{proof}

The next step is to use the plumpness condition (F2) in order to locate a
sufficiently large cube inside each $B_m$. Namely, for each $m=1,\ldots,N$, there
is $z_m\in \bar B(\omega_m,  2^{-n-k}/3)$ such that
\[
B(z_m, \kappa  2^{-n-k}/3) \subset D\,.
\]
Let us consider a Whitney cube $Q_m\in \mathcal{W}(D)$ for which $z_m\in Q_m$.
By inequalities \eqref{dist_est}, we have $Q_m \subset B_m$. Moreover,
\[
\kappa  2^{-n-k}/3\le \mathrm{dist}(z_m,\partial D) \le 5 \mathrm{diam}(Q_m)\le 5\mathrm{dist}(z_m,\partial D)\le 5\cdot 2^{-n-k}/3\,.
\]
Hence, by our definition \eqref{e.tau_def} of $\tau$, we obtain
\[
\tau^{-1} \le 2^{d(n+k)}\lvert Q_m\rvert \le \tau\,.
\]
That is, cube $Q_m\subset B_m$ belongs to $\mathcal{B}^{(n+k)}$. Since the balls $B_m$, $m=1,\ldots,N$, are
disjoint, also the cubes $Q_m$ are disjoint. Hence, the indexing set
\[
V(B^{(n)}_j,k ) =  \{i\,:\, B^{(n+k)}_i=Q_m \text{ for some } m=1,\ldots,N^{(n,k)}_j\}
\]
contains exactly $N_j^{(n,k)}$ indices.

\subsection{Proof of Proposition \ref{p.fat}}

We focus on conditions (B3) and (B4), as the remaining conditions are clearly satisfied.
Let us fix $B^{(n)}_j$ and $k$ large enough so that
$V(B^{(n)}_j,k)$ is defined. Let us consider $i\in V(B^{(n)}_j,k)$, and two given points
$x\in B^{(n)}_j$ and $y\in B_i^{(n+k)}$. Using the notation above, we have
$B_i^{(n+k)} = Q_m\subset B_m$ for some $m=1,\ldots,N_j^{(n,k)}$.
Thus,
\begin{align*}
\lvert x-y \rvert &\le \lvert x - x_j^{(n)}\rvert + \lvert x_j^{(n)} - y_j^{(n)}\rvert + \lvert y_j^{(n)} - \omega_m\rvert + \lvert \omega_m - y\rvert\\
&< \mathrm{diam}(B^{(n)}_j)  + \mathrm{dist}(x_j^{(n)},\partial D) + 2^{-n} + 2^{-n-k} \\&\le 8 \sqrt d \tau^{1/d}2^{-n}\,.
\end{align*}
This is condition (B3).
A particular consequence of this estimate is the following. We fix a cube $ B^{(n+k)}_i$ and a point $y$ therein. Then, if $B^{(n)}_j\in \mathcal{B}^{(n)}$ is such
that $i\in V(B^{(n)}_j,k)$,
\[
B^{(n)}_j \subset B(y,8 \sqrt d \tau^{1/d}2^{-n})\,.
\]
Since the interiors of cubes in $\mathcal{B}^{(n)}$ are disjoint, we find that there are at most
\[
\frac{(16 \sqrt d \tau^{1/d}2^{-n})^d}{\tau^{-1}2^{-nd}} = (16 \sqrt d )^d \tau^2
\]
cubes $B^{(n)}_j$ subject to the conditions above. By using this fact, we may now deduce the remaining estimate as follows;
For a fixed $i$,
\begin{align*}
\sum_{j:i\in V(B^{(n)}_j,k)}  \frac{1}{\sharp V(B_j^{(n)},k)} =\sum_{j:i\in V(B^{(n)}_j,k)} \frac{1}{N_j^{(n,k)}} \le (16 \sqrt d )^d \tau^{2+\lambda/d} C^{-1} 5^\lambda 2^{-k\lambda}\,.
\end{align*}
This is condition (B4).
\qed

\subsection{Lower Assouad dimension and uniform fatness}\label{s.unif_fatness}
We provide a
useful connection between the lower Aikawa dimension and (local) uniform fatness.
For further discussion, we refer to \cite{KLV}.
Uniform fatness is
usually defined in terms of Riesz capacities, \cite{AH,Lewis1988}. 
In case of closed sets, there is an equivalent definition---in terms of Hausdorff content---that we adopt. This equivalence
is based on the self-improving properties of closed uniformly fat sets, \cite{ihnatsyeva3}.

Recall that the \emph{$\lambda$-Hausdorff content} of a set $E \subset \R^d$ is
\[
\mathcal{H}^\lambda_\infty(E)=\inf\bigg\{\sum_{i=1}^{\infty}r_i^\lambda :
 E\subset\bigcup_{i=1}^\infty B(x_i,r_i),\ r_i>0 \bigg\}.
\]
As is easily seen, we may allow also finite coverings in the infimum above. 
Let $1<p<\infty$ and $0<s<d/p$.
We say that the boundary  $\partial D$ is {\em $(s,p)$-uniformly fat}, if there is $d-sp<\lambda\le d$ 
and a constant $C>0$ such that
\begin{equation}\label{e.hinfty}
\mathcal{H}^\lambda_\infty(B(x,R)\cap \partial D)\ge CR^\lambda
\end{equation}
for all $x\in \partial D$ and $R>0$. 

Note that  $\partial D$ and $D$ have to be unbounded
if the boundary is $(s,p)$-uniformly fat.
Remark~2.3 in \cite{KLV} shows that
 $\underline{\mathrm{dim}}_A(\partial D)$ is the supremum 
of all $\lambda\ge 0$ for which \eqref{e.hinfty} holds for every $x\in \partial D$ and $0<R<\mathrm{diam}(\partial D)$.
Below, for the convenience of the reader, we provide a detailed treatment
of certain consequences of this  statement---that are needed for Corollary \ref{t.fat}.

\begin{prop}\label{p.unif_fatness}
Suppose  $D$ is an open set in  $\R^d$ so that $\partial D$
 is $(s,p)$-uniformly fat for $1<p<\infty$ and $0<s<d/p$.  Then $D$ satisfies
 condition (F1) for some $d-sp<\lambda \le d$ and, as a consequence,
 we have a strict inequality $\underline{\mathrm{dim}}_A(\partial D)>d-sp$.
\end{prop}

\begin{proof}
By assumption, there is $d-sp<\lambda \le d$ and $C>0$ such that
\eqref{e.hinfty} holds for 
all $x\in\partial D$ and $R>0$.
Let us fix $x\in \partial D$ and  $0<r<R$. Suppose that $B_1,\ldots,B_N$ is
a cover of $B(x,R)\cap \partial D$ by balls $B_j=B(\omega_j,r)$ with $\omega_j\in\partial D$.
Then, by \eqref{e.hinfty},
\[
Nr^\lambda = \sum_{j=1}^N r^\lambda \ge \mathcal{H}^{\lambda}_\infty(B(x,R)\cap \partial D)\ge CR^\lambda\,.
\]
Thus, $N\ge C(R/r)^\lambda$, as required.
\end{proof}

As we have observed, uniform fatness is a convenient notion in case of unbounded open sets.
In case of a bounded open set $D$ in $\R^d$, it is
natural to assume that $\partial D$ is {\em $(s,p)$-locally uniformly fat.}
That is, there is $d-sp < \lambda \le d$ and a constant $C>0$
such that inequality \eqref{e.hinfty} holds for all $x\in \partial D$ and $0<R<2\diam(\partial D)<\infty$.

The following result is analogous to Proposition \ref{p.unif_fatness}.

\begin{prop}\label{p.local_unif_fatness}
Let  $D$ be a bounded open set in  $\R^d$ such that $\partial D$
 is $(s,p)$-locally uniformly fat for $1<p<\infty$ and $0<s<d/p$.  Then $D$ satisfies
 condition (F1) for some $d-sp<\lambda \le d$ and, as a consequence,
 we have a strict inequality $\underline{\mathrm{dim}}_A(\partial D)>d-sp$.
\end{prop}

\begin{exmp}
Consider the Koch snowflake domain
$D\subset \R^2$. It is a bounded $\kappa$-plump domain and $\partial D$
is  $(s,p)$-locally uniformly fat if $1<p<\infty$ and $2-\log 4/ \log 3 < sp <2 $, see e.g. \cite{ihnatsyeva3}.
\end{exmp}

\section{Thin boundary}\label{s.thin}

The main result in this section is the following.

\begin{prop}\label{p.thin}
 Let $D\not=\emptyset$ be an unbounded $\kappa$-plump open set in  $\R^d$, $D\neq \R^d$.
Then $D$ satisfies condition $\DC(a,\lambda,d)$ for $a=2$ and $\lambda>\overline{\mathrm{dim}}_A(\partial D)$.
The associated constant $M$ depends only on $d$, $\kappa$, $\lambda$ and the constant $C$ appearing in (T1) below.
\end{prop}

Before  the proof, let us clarify the assumptions.
Under the assumptions of Proposition \ref{p.thin}, the following two
conditions (T1) and (T2) hold.

\begin{itemize}
\item[(T1)] 
there is a constant $C> 0$ as follows.
Assuming that $0<r< R$ and $x\in \partial D$, there is a cover of $B(x,R)\cap \partial D$
by using balls $B(\omega_j,r)$ with $\omega_j\in \partial D$, $j=1,\ldots,N$, such that the number
of these balls satisfies inequality $N\le C(R/r)^\lambda$.
\item [(T2)] for each $0<r$ and each $x\in \partial D$, there
is a point $z\in \bar B(x,r)$ such that $B(z,\kappa r)\subset D$.
\end{itemize}

We note that in Definition~\ref{def:uAssouad} the restriction $R<2\diam(E)$ may be removed,
resulting in no such restriction in (T1).

\subsection{Construction of families $\mathcal{B}^{(n)}$}\label{s.b_sets}
Let $\tau>1$ be defined by \eqref{e.tau_def}.
For $n\in \Z$, we define
\[
\mathcal{B}^{(n)}:= \mathcal{B}^{(n)}_{2}:=\{B^{(n)}_j\} := \{Q\in\mathcal{W}(D)\,:\, \tau^{-1}\le 2^{-nd} \lvert Q\rvert \le \tau\}\,.
\]
For properties of cubes in these families and definition of $y_j^{(n)}$, see
 \S \ref{s.a_sets}.
Recall also that $\mathcal{W}(D)$ is a Whitney decomposition of $D$, 
we refer to \S \ref{s.notation}.

\subsection{Families $V(B^{(n)}_j,k)$ with $k>0$}
Let us fix a cube $B^{(n)}_j\in\mathcal{B}^{(n)}$. 
By condition (T2), there is a point $z_j^{(n)}\in \bar B(y_j^{(n)}, 2^{n+k})$ such that
\[
B(z_j^{(n)}, \kappa 2^{n+k})\subset D\,.
\]
Observe how the unboundedness of $D$ is visible here,
as $k>0$ is arbitrary.
Let $Q_j\in\mathcal{W}(D)$ be a Whitney cube such that $z_j^{(n)}\in Q_j$. Then
\begin{align*}
\kappa\cdot 2^{n+k} \le \mathrm{dist}(z_j^{(n)}, \partial D)\le 5 \mathrm{diam}(Q_j) \le 5\mathrm{dist}(z_j^{(n)}, \partial D) \le 5 \cdot 2^{n+k}\,.
\end{align*}
By definition of \eqref{e.tau_def} of $\tau$,
we then have $Q_j= B^{(n+k)}_i\in \mathcal{B}^{(n+k)}$ for some index $i$.
We define
\[
V(B^{(n)}_j, k ) = \{i\}\,.
\]

\subsection{Proof of Proposition \ref{p.thin}}
We need an auxiliary estimate
analogous to \cite[Lemma 4.3]{KLV};
condition (T1) is our primary tool.
For  $m\in\Z$, $\omega\in \partial D$, and $R>0$, we denote
\[
\mathcal{W}_m(D;B(\omega,R)) = \{Q\in \mathcal{W}_m(D)\,:\, Q\subset B(\omega,R)\}\,.
\]

\begin{lem}
Let $0<2^{-m}\le R$, where $m\in \Z$,
and let $\lambda>\overline{\mathrm{dim}}_A(\partial D)$.
Then for every $\omega\in \partial D$,
\[
\sharp \mathcal{W}_m(D; B(\omega,R))\le C (14\sqrt d)^{d+\lambda}\bigg(\frac{R}{2^{-m}}\bigg)^\lambda\,,
\]
where $C$ is as in condition (T1).
\end{lem}

\begin{proof}
Suppose $B_1,\ldots,B_N$ is a cover of  $B(\omega,6\sqrt dR)\cap \partial D$ by balls $B_j=B(\omega_j,2^{-m})$
that are
centred in $\partial D$, see condition (T1). Consider a cube $Q\in\mathcal{W}_m(D;B(\omega,R))$,
and fix a point $y_Q\in\partial D$ such that $\lvert x_Q-y_Q\rvert = \mathrm{dist}(x_Q,\partial D)$. Here
$x_Q$ denotes the midpoint of $Q$.
By inequalities \eqref{dist_est} and the fact that $Q\subset B(\omega,R)$,
\begin{align*}
\lvert y_Q -\omega\rvert \le \lvert y_Q-x_Q\rvert + \lvert x_Q-\omega\rvert < 5\mathrm{diam}(Q)+R \le 6\sqrt d R\,.
\end{align*}
By the covering property, there is $j=j(Q)$ such that $y_Q\in B_j$. 
We can infer that
\[\mathcal{W}_m(D;B(\omega,R)) = \bigcup_{j=1}^N \mathcal{Q}_j\,,\] where
 $\mathcal{Q}_j=\{Q\in\mathcal{W}_m(D;B(\omega,R))\,:\, y_Q\in B_j\}$.
Let $Q\in\mathcal{Q}_j$. Then, for every $x\in Q$, 
\begin{align*}
\lvert x-\omega_j \rvert &\le \lvert x-x_Q\rvert + \lvert x_Q-y_Q\rvert + \lvert y_Q-\omega_j\rvert\\
&< \mathrm{diam}(Q) + \mathrm{dist}(x_Q,\partial D) + 2^{-m} \\&\le 7\sqrt d 2^{-m}\,.
\end{align*}
Since the interiors of cubes in the family $\mathcal{Q}_j$ are disjoint, there are at most
\[
\frac{\lvert B(\omega_j, 7\sqrt d 2^{-m})\rvert}{ 2^{-md}} \le (14 \sqrt d)^d
\]
cubes in this family. Hence,
\begin{align*}
\sharp \mathcal{W}_m(D;B(\omega,R)) &\le \sum_{j=1}^N \sharp \mathcal{Q}_j \le (14\sqrt d)^d N 
\le C\bigg(\frac{6\sqrt d R}{2^{-m}}\bigg)^\lambda \cdot (14 \sqrt d)^d\,.
\end{align*}
This concludes the proof.
\end{proof}

We are ready to prove the main result in this section.

\begin{proof}[Proof of Proposition \ref{p.thin}]
The properties (B1) and (B2) are clear. 
In order to verify condition (B3),
let us fix $k>0$ and a cube $B^{(n)}_j$.
Consider
$i\in V(B^{(n)}_j,k)$, and points $x\in B^{(n)}_j$ and $z\in B^{(n+k)}_i$. Then,
by the construction above,
\begin{align*}
\lvert z-x\rvert &\le \lvert z-z_j^{(n)}\rvert + \lvert z_j^{(n)} - y_j^{(n)}\rvert + \lvert y_j^{(n)}- x_j^{(n)}\rvert + \lvert x_j^{(n)}-x\rvert\\
&\le \mathrm{diam}(B^{(n+k)}_i) + 2^{n+k} + \mathrm{dist}(x_j^{(n)}, \partial D) + \mathrm{diam}(B^{(n)}_j)
<  8\sqrt d \tau^{1/d} 2^{n+k}\,.
\end{align*}
This is condition (B3). 

In order to verify the last condition (B4), we fix  cubes $B^{(n+k)}_i$
and $B^{(n)}_j$ such that $i\in V(B^{(n)}_j, k)$. Then $\sharp V(B^{(n)}_j,k)=1$.
Moreover,
\[
B_j^{(n)}\subset B(y_i^{(n+k)}, 13\sqrt d \tau^{1/d} 2^{n+k})\,.
\]
Indeed, for any $x\in B^{(n)}_j$,
\[
\lvert x-y_{i}^{(n+k)}\rvert \le \lvert x-x_i^{(n+k)}\rvert + \lvert x_i^{(n+k)} - y_i^{(n+k)}\rvert < 13\sqrt d \tau^{1/d} 2^{n+k}\,.
\]
We still need another auxiliary estimate, namely, if $m\in\Z$ is such that $B^{(n)}_j\in\mathcal{W}_m(D)$, then
$\tau^{-1/d}\le 2^{m+n}\le \tau^{1/d}$. We can finally proceed as follows
\begin{align*}
\sum_{j:i\in V(B^{(n)}_j,k)} \frac{1}{\sharp V(B^{(n)}_j,k)} &= \sharp \{j\,:\, i\in V(B^{(n)}_j,k)\}\\
&= \sum_{m} \sharp \{j\,:\, i\in V(B^{(n)}_j,k)\text{ and } B^{(n)}_j\in\mathcal{W}_m(D)\}\\
&\le \sum_{m} \sharp \mathcal{W}_m(D;  B(y_i^{(n+k)},13\sqrt d \tau^{1/d} 2^{n+k}))\\
&\le \sum_{m} C (14\sqrt d)^{d+\lambda}\bigg(\frac{13\sqrt d \tau^{1/d} 2^{n+k}}{2^{-m}}\bigg)^\lambda\,,
\end{align*}
where $m$ ranges over indices $-n-\log_2 \tau^{1/d} \le m \le -n+\log_2 \tau^{1/d}$. 
This yields condition (B4).
\end{proof}

%\bibliographystyle{abbrv}
%\bibliography{fh}

\def\cprime{$'$}

\end{document}